\documentclass[leqno,12pt]{amsart}
\usepackage{amsmath,graphicx,amssymb}
\usepackage[all]{xy}
\usepackage{epic}
\usepackage{eepic}
\usepackage{hyperref}

\newtheorem*{thmA}{Theorem A}
\newtheorem*{thmB}{Theorem B}
\newtheorem*{thmC}{Theorem C}
\newtheorem*{thmD}{Theorem D}
\newtheorem*{thmE}{Theorem E}
\newtheorem*{prop11}{Proposition 11}
\newtheorem*{cor12}{Corollary 12}
\newtheorem*{prop13}{Proposition 13}
\newtheorem*{cor14}{Corollary 14}
\newtheorem{proposition}{Proposition}
\newtheorem{cor}[proposition]{Corollary}
\newtheorem{lemma}[proposition]{Lemma}
\newtheorem{remark}[proposition]{Remark}

\renewcommand{\AA}{\mathcal A}

\newcommand{\N}{\mathbb N}

\newcommand{\C}{\mathbb C}

\newcommand{\F}{\mathbb F}

\newcommand{\BB}{\mathcal B}
\newcommand{\HH}{\mathcal H}

\newcommand{\KK}{\mathcal K}

\newcommand{\comment}[1]{}

\title[A note on full free product $C^*$-algebras]{A note on full free product $C^*$-algebras, \\  lifting and quasidiagonality}

\author{Florin P. Boca}

\address{University of Wales Swansea, Department of Mathematics,
Swansea SA2 8PP, U.K.}

\address{On leave from:}

\address{Institute of Mathematics of the Romanian Academy,
P.O. Box 1-764, 70700 Bucharest, ROMANIA}


\begin{document}


\begin{abstract}
We study lifting properties for full product $C^*$-algebras with amalgamation over $\C 1$ and give new proofs for some results of Kirchberg and Pisier. We extend the
result of Choi on the quasidiagonality of $C^* (\F_n)$, proving that the free product with amalgamation over $\C 1$ of a
family of unital quasidiagonal $C^*$-algebras is quasidiagonal.
\end{abstract}

\maketitle

All $C^*$-algebras in this note are unital. By representations and respectively morphisms we mean $\ast$-representations of $C^*$-algebras on Hilbert spaces,
respectively $\ast$-morphisms between $C^*$-algebras. If $A$ and $B$ are $C^*$-algebras, $A\odot B$ denotes the tensor product of $A$ and $B$ viewed as a $\ast$-algebra,
$A\underset{\operatorname{min}}{\otimes} B$ the minimal $C^*$-tensor product and $A\underset{\operatorname{max}}{\otimes} B$ the maximal one. If $M$ is a von Neumann algebra,
$A\underset{\operatorname{nor}}{\otimes} M$ denotes the normal $C^*$-tensor product (see \cite{Was}).

Throughout this paper, $F$ will denote a free group (not necessarily countable) and $C^*(F)$ its full group $C^*$-algebra.

Let $A$ and $B$ be $C^*$-algebras, $J$ be an ideal of $B$ and $\pi:B \rightarrow B/J$ the canonical quotient morphism. A unital completely positive
(ucp) map $\Phi:A \rightarrow B/J$ is \emph{liftable} if there exists a ucp map $\widetilde{\Phi}:A\rightarrow B$ such that
$\pi\widetilde{\Phi} =\Phi$ and \emph{locally liftable} if for any finite dimensional operator system $X\subset A$, there exists a ucp
map $\widetilde{\Phi}_X:X \rightarrow B$ such that $\pi\widetilde{\Phi}_X =\Phi\vert_X$. A $C^*$-algebra has the \emph{lifting property} (LP),
respectively the \emph{local lifting property} (LLP) if any ucp map $\Phi:A \rightarrow B/J$ is liftable, respectively locally liftable, for any
$C^*$-algebra $B$ any any ideal $J$ of $B$.

Significant connections between these lifting properties and $C^*$-tensor products have been established by Kirchberg who proved the following two important results:

\begin{thmA}[\cite{Kir2}, Theorem 1.1]
$A\underset{\operatorname{nor}}{\otimes} M=A\underset{\operatorname{max}}{\otimes} M$ for any von Neumann algebra $M$ and any separable $C^*$-algebra $A$ with {\rm LP}.
\end{thmA}

\begin{thmB}[\cite{Kir1}, Proposition 2.2]
If $A$ is a $C^*$-algebra, then $A\underset{\operatorname{min}}{\otimes} \BB (\HH)=A\underset{\operatorname{max}}{\otimes} \BB (\HH)$ if and only if $A$ has {\rm LLP}.
\end{thmB}

In a recent paper (\cite{Pis}, Theorem 0.1 and 1.24) Pisier gave direct and elegant proofs of:

\begin{thmC}
$C^*(F)\underset{\operatorname{min}}{\otimes} \BB(\HH)=C^*(F)\underset{\operatorname{max}}{\otimes} \BB(\HH)$.
\end{thmC}

\begin{thmD}
$C^*(F)\underset{\operatorname{nor}}{\otimes} M=C^*(F)\underset{\operatorname{max}}{\otimes} M$ for any von Neumann algebra $M$.
\end{thmD}

Furthermore, he extended Theorem A proving

\begin{thmE}[\cite{Pis}, Theorem 0.2]
Let $(A_i)_{i\in I}$ be a family of $C^*$-algebras such that
\begin{equation*}
A_i \underset{\operatorname{min}}{\otimes} \BB(\HH) =A_i \underset{\operatorname{max}}{\otimes} \BB(\HH),\quad \forall i\in I .
\end{equation*}
Then $A\underset{\operatorname{min}}{\otimes} \BB(\HH)=A\underset{\operatorname{max}}{\otimes} \BB(\HH)$, where $A=\underset{i\in I}{\ast} A_i$
denotes the full free $C^*$-product with amalgamation over $\C 1$.
\end{thmE}

This note grew up in an attempt to understand the connection between these results and simplify the proofs of Kirchberg's
deep results. In the first part we give new proofs for Theorems A and B, relying on Theorems D respectively C and on standard arguments from \cite{EH}
(see also \cite{Was}).

The proof of Theorem E uses a factorization type result of Paulsen and Smith (\cite{PS}) and the Christensen-Effros-Sinclair  embedding of the Haagerup
tensor product of two $C^*$-algebras into their full $C^*$-amalgamated product (\cite{CES}). We give a different proof of Theorem E using the equivalence
between the conditions in Theorem B and the local liftability of $\operatorname{id}_A$ and the extension result for ucp maps on full free product
$C^*$-algebras from \cite{Bo1}.

As a consequence of our proof of Theorem A, if $(A_i)_{i\in I}$ is a family of $C^*$-algebras such that $\operatorname{id}_{A_i}$ is liftable
for all $i\in I$, then
\begin{equation*}
\Big( \underset{i\in I}{\ast} A_i \Big) \underset{\operatorname{nor}}{\otimes} M =
\Big( \underset{i\in I}{\ast} A_i \Big) \underset{\operatorname{max}}{\otimes} M
\end{equation*}
for any von Neumann algebra $M$.

The existence of a faithful block-diagonal $\ast$-representation of $C^*(\F_2)$ on a separable Hilbert space (\cite{Cho}) implies
that $C^*(F)$ is quasidiagonal for any free group $F$ (countable or not). We extend this result by proving that if $(A_i)_{i\in I}$ is a family of
(unital) quasidiagonal $C^*$-algebras, then $\underset{i\in I}{\ast} A_i$ is quasidiagonal. The proof relies
on the characterization of quasidiagonality with approximate multiplicative ucp maps (\cite{Voi}) and on \cite{Bo1}.

The following lemma is probably well-known:

\begin{lemma}
Let $J$ be an ideal of a $C^*$-algebra $B$, $\HH$ be a Hilbert space and $M$ a von Neumann algebra. Then:
\begin{itemize}
\item[(i)] If $B\underset{\operatorname{min}}{\otimes} \BB(\HH)=B \underset{\operatorname{max}}{\otimes} \BB(\HH)$, then
$J\underset{\operatorname{min}}{\otimes} \BB(\HH)=J\underset{\operatorname{max}}{\otimes} \BB(\HH)$.
\item[(ii)] If $B\underset{\operatorname{nor}}{\otimes} M=B\underset{\operatorname{max}}{\otimes} M$, then $J\underset{\operatorname{nor}}{\otimes} M =J\underset{\operatorname{max}}{\otimes} M$.
\end{itemize}
\end{lemma}

\begin{proof}
(i) Let $x=\sum\limits_i a_i\otimes x_i \in J \odot \BB(\HH)$. Using the fact that any (nondegenerate) representation
$\rho$ of $J$ extends (uniquely) to a representation $\pi$ of $B$ on the same Hilbert space such that $\rho(J)$ is strongly dense
in $\pi(B)$ (\cite{Dix}, Proposition 2.10.4), the hypothesis and the isometric embedding of
$J\underset{\operatorname{min}}{\otimes} \BB(\HH)$ into $B\underset{\operatorname{min}}{\otimes} \BB(\HH)$ (\cite{Was}), we obtain:
\begin{equation*}
\begin{split}
\| x\|_{J\underset{\operatorname{max}}{\otimes} \BB(\HH)} & = \sup\Bigg\{ \bigg\| \sum_i \pi_1(a_i)\pi_2(x_i)\bigg\|;
\begin{aligned} & \pi_2\ \mbox{\rm representation of} \ \BB(\HH) \\  & \pi_1:J\rightarrow \pi_2(\BB(\HH))^\prime\ \mbox{\rm morphism}\end{aligned} \Bigg\}  \\
& =\sup \Bigg\{ \bigg\| \sum_i \widetilde{\pi}_1 (a_i) \pi_2(x_i) \bigg\| ;
\begin{aligned} & \pi_2 \ \mbox{\rm representation of}\ \BB(\HH) \\
& \widetilde{\pi}_1:B\rightarrow \pi_2 (\BB(\HH))^\prime\ \mbox{\rm morphism} \end{aligned} \Bigg\}  \\
& =\| x\|_{B\underset{\operatorname{max}}{\otimes} \BB(\HH)} =\| x\|_{B\underset{\operatorname{min}}{\otimes} \BB(\HH)}
=\| x\|_{J\underset{\operatorname{min}}{\otimes} \BB(\HH)} .
\end{split}
\end{equation*}

(ii) Let $x=\sum\limits_i a_i \otimes x_i \in J\odot M$. As in (i) we get
$\| x\|_{J\underset{\operatorname{max}}{\otimes} M} =\| x\|_{B\underset{\operatorname{max}}{\otimes} M}$. By hypothesis
$\| x\|_{B\underset{\operatorname{max}}{\otimes} M}=\| x\|_{B\underset{\operatorname{nor}}{\otimes} M}$ and therefore:
\begin{equation*}
\begin{split}
\| x\|_{J\underset{\operatorname{nor}}{\otimes} M} & = \sup\Bigg\{ \bigg\| \sum_i \pi_1(a_i) \pi_2(x_i) \bigg\| ;
\begin{aligned} & \pi_2 \ \mbox{\rm normal representation of}\ M \\
& \pi_1 :J \rightarrow \pi_2(M)^\prime\ \mbox{\rm morphism} \end{aligned}\Bigg\} \\
& =\sup\Bigg\{ \bigg\| \sum_i \widetilde{\pi}_1 (a_i) \pi_2(x_i)\bigg\| ;
\begin{aligned} & \pi_2 \ \mbox{\rm normal representation of} \ M \\
& \widetilde{\pi}_1 :B \rightarrow \pi_2(M)^\prime \  \mbox{\rm morphism} \end{aligned} \Bigg\} \\
& =\| x\|_{B\underset{\operatorname{nor}}{\otimes} M} .\qedhere
\end{split}
\end{equation*}
\end{proof}

\begin{remark} {\rm  (i) If $B$ is a $C^*$-algebra, $J$ an ideal of $B$ and $M$ a von Neumann algebra, then
$J\underset{\operatorname{nor}}{\otimes} M$ is an ideal of $B\underset{\operatorname{nor}}{\otimes} M$.

(ii)  If $A_1$ and $A_2$ are $C^*$-algebras and $\Phi_j:A_j \rightarrow \BB(\HH)$ are ucp maps such that
$[\Phi_1(a_1) , \Phi_2(a_2) ] = 0$, $a_j\in A_j$, there exist a Hilbert space $\KK$, an isometry $V:\HH\rightarrow \KK$ and
representations $\pi_j:A_j \rightarrow \BB(\KK)$ such that $\Phi_1(a_1)\Phi_2(a_2)=V^* \pi_1(a_1)\pi_2(a_2)V$,
$a_j\in A_j$, $j=1,2$ (see e.g. Proposition 4.23 in \cite{Tak} or Theorem 1.6 in \cite{Was}). This shows that
$\Phi \Big( \sum\limits_i a_i \otimes b_i \Big) =\sum\limits_i \Phi_1(a_i) \Phi_2(b_i)$ extends to a ucp map
$\Phi:A_1 \underset{\operatorname{max}}{\otimes} A_2 \rightarrow \BB(\HH)$.

An inspection of the proof of Theorem 1.6 in \cite{Was} shows actually that if one of the ucp maps $\Phi_j$ is normal,
the representation $\pi_j$ is also normal. In particular if $A$ is a $C^*$-algebra, $M$ a von Neumann algebra,
$\Phi_1:A\rightarrow \BB(\HH)$, $\Phi_2:M\rightarrow \BB(\HH)$ are ucp maps and $\Phi_2$ is normal, then $\Phi$ extends to a
ucp map $\Phi:A \underset{\operatorname{nor}}{\otimes} M \rightarrow \BB(\HH)$. }
\end{remark}

\begin{lemma}
Let $J$ be an ideal in a $C^*$-algebra $B$ and $A=B/J$. Then:

{\rm (i)} If $\operatorname{id}_A :A\rightarrow B/J$ is locally liftable, then the $C^*$-norm $\nu_{\operatorname{min}}$ induced
on $A\odot C$ from the quotient $\big( B\underset{\operatorname{min}}{\otimes} C\big) \big/ \big( J\underset{\operatorname{min}}{\otimes} C\big)$
coincides with $\underset{\operatorname{min}}{\otimes}$ for any $C^*$-algebra $C$.

{\rm (ii)} If $\operatorname{id}_A:A\rightarrow B/J$ is liftable, then the $C^*$-norm $\nu_{\operatorname{nor}}$ induced on
$A\odot M$ from $\big( B\underset{\operatorname{nor}}{\otimes} M\big) \big/ \big( J \underset{\operatorname{nor}}{\otimes} M \big)$ coincides
with $\underset{\operatorname{nor}}{\otimes}$ for any von Neumann algebra $M$.
\end{lemma}

\begin{proof}
(i) The proof is as in Theorem 3.2, ${\rm (i)}\Rightarrow {\rm (ii)}$, \cite{EH}. Denote by $\pi$ the quotient morphism
from $B$ onto $A$. Let $x=\sum\limits_i \pi(b_i)\otimes c_i$, with $b_i\in B$, $c_i\in C$ and $X\subset A$ be a finite dimensional operator
system which contains $\pi(b_i)$. By hypothesis, there exists a ucp map $\Phi:X \rightarrow B$ such that $\pi\Phi=\operatorname{id}_X$. We have:
\begin{equation*}
\| x\|_{\nu_{\operatorname{min}}} =\underset{a\in J\underset{\operatorname{min}}{\otimes} C}{\inf} \big\| \big( \Phi\otimes \operatorname{id}_C \big) (x) +a\big\|_{B\underset{\operatorname{min}}{\otimes} C}
\leqslant \big\| \big(\Phi\otimes \operatorname{id}_C \big) (x) \big\|_{B\underset{\operatorname{min}}{\otimes} C} \leqslant
\| x\|_{A\underset{\operatorname{min}}{\otimes} C} .
\end{equation*}

The inequality $\| x\|_{A\underset{\operatorname{min}}{\otimes} C} \leqslant \| x\|_{\nu_{\operatorname{min}}}$ is clear.

(ii) The inequality $\| x\|_{\nu_{\operatorname{nor}}} \leqslant \| x\|_{A\underset{\operatorname{nor}}{\otimes} M}$, $x\in A \odot M$ follows as above, considering a ucp lifting
$\Phi:A\rightarrow B$ of $\operatorname{id}_A$ and using the fact that $\Phi\otimes \operatorname{id}_M$ extends by Remark 2 to a contractive map
$\Phi\otimes \operatorname{id}_M :A\underset{\operatorname{nor}}{\otimes} M \rightarrow B \underset{\operatorname{nor}}{\otimes} M$.

Since $\pi\otimes \operatorname{id}_M : B\underset{\operatorname{nor}}{\otimes} M \rightarrow A \underset{\operatorname{nor}}{\otimes} M$ is contractive and its kernel
contains $J\underset{\operatorname{nor}}{\otimes} M$, we get for all $x=\sum\limits_i b_i \otimes c_i \in A \odot M$:
\begin{equation*}
\begin{split}
\| x\|_{A\underset{\operatorname{nor}}{\otimes} M} & = \underset{a\in J \underset{\operatorname{nor}}{\otimes} M}{\inf}
\bigg\| \big( \pi\otimes \operatorname{id}_M \big) \bigg( \sum_i \Phi(b_i) \otimes c_i +a\bigg) \bigg\|_{A\underset{\operatorname{nor}}{\otimes} M}  \\
& \leqslant \underset{a\in J\underset{\operatorname{nor}}{\otimes} M}{\inf}
\bigg\| \sum_i \Phi(b_i)\otimes c_i +a\bigg\|_{B\underset{\operatorname{nor}}{\otimes} M} =\| x\|_{\nu_{\operatorname{nor}}} .\qedhere
\end{split}
\end{equation*}
\end{proof}

\begin{cor}
Let $A$ be a $C^*$-algebra. Then:
\begin{itemize}
\item[(i)]
If $\operatorname{id}_A$ is locally liftable, then $A\underset{\operatorname{min}}{\otimes} \BB(\HH) =A \underset{\operatorname{max}}{\otimes} \BB(\HH)$.
\item[(ii)]
If $\operatorname{id}_A$ is liftable, then $A\underset{\operatorname{nor}}{\otimes} M=A\underset{\operatorname{max}}{\otimes} M$ for any von Neumann algebra $M$.
\end{itemize}
\end{cor}

\begin{proof}
(i) Let $F$ be a freee group (not necessarily countable) and $\pi:C^*(F)\rightarrow A$ a morphism which is onto and $J=\operatorname{Ker}\pi$.
The following sequence is exact (\cite{Was}):
\begin{equation}\label{eq1}
0\rightarrow J\underset{\operatorname{max}}{\otimes} \BB(\HH) \rightarrow C^*(F) \underset{\operatorname{max}}{\otimes} \BB(\HH) \rightarrow
A\underset{\operatorname{max}}{\otimes} \BB(\HH) \rightarrow 0.
\end{equation}

If $\operatorname{id}_A$ is liftable then $A\underset{\operatorname{min}}{\otimes} \BB(\HH) =A\underset{\nu_{\operatorname{min}}}{\otimes} \BB(\HH)$
by Lemma 3. Since $C^*(F) \underset{\operatorname{max}}{\otimes} \BB(\HH) = C^*(F) \underset{\operatorname{min}}{\otimes} \BB(\HH)$
(\cite{Kir2}, see Theorem 0.1, \cite{Pis} for a short proof), Lemma 1 shows that
$J\underset{\operatorname{max}}{\otimes} \BB(\HH) = J \underset{\operatorname{min}}{\otimes} \BB(\HH)$. By the exactness of \eqref{eq1} we have
the following canonical equalities:
\begin{equation*}
\begin{split}
A\underset{\operatorname{max}}{\otimes} \BB(\HH) & =\big( C^*(F)\underset{\operatorname{max}}{\otimes} \BB(\HH)\big) \big/ \big(
J\underset{\operatorname{max}}{\otimes} \BB(\HH)\big) =\big( C^*(F) \underset{\operatorname{min}}{\otimes} \BB(\HH)\big) \big/
\big( J\underset{\operatorname{min}}{\otimes} \BB(\HH)\big) \\
& =A\underset{\nu_{\operatorname{min}}}{\otimes} \BB(\HH) = A\underset{\operatorname{min}}{\otimes} \BB(\HH) .
\end{split}
\end{equation*}

(ii) For any free group $F$ and von Neumann algebra $M$ we have
\begin{equation}\label{eq2}
C^*(F) \underset{\operatorname{max}}{\otimes} M=C^*(F) \underset{\operatorname{nor}}{\otimes} M .
\end{equation}

The subsequent proof of \eqref{eq2} is from \cite{Pis} and the idea is to consider as in \cite{Haa}
for any $x_1,\ldots, x_n \in M$ the linear (completely bounded) map $T=T_{x_1,\ldots,x_n}:\C^n \rightarrow M$,
$Te_j=x_j$. If $\sigma$-finite, $M$ can be represented on a Hilbert space with separating vector and according to
Lemma 3.5, \cite{Haa} we have
\begin{equation}\label{eq3}
\| T\|_{\operatorname{dec}} =\sup\limits_{\substack{a_j \in M^\prime \\ \| a_j\| \leqslant 1}} \bigg\| \sum_{j=1}^n a_j x_j \bigg\| .
\end{equation}

Denote by $\{ U_j\}_j$ the canonical generators of $C^*(F)$. Let $x_j\in M$ such that $x_j \neq 0$ only for finitely many $j'$s and
$\big\| \sum\limits_j U_j \otimes x_j \big\|_{C^*(F)\underset{\operatorname{nor}}{\otimes} M} \leqslant 1$. Let $a_j\in M^\prime$ such that
$\| a_j\| \leqslant 1$. By \cite{Boz} the map $\Phi(U_j)=a_j$ extends to a ucp map $\Phi:C^*(F)\rightarrow M^\prime$ and according to \eqref{eq3} and Remark 2, we obtain
\begin{equation*}
\begin{split}
\big\| T_{\{ x_j\}_j} \big\|_{\operatorname{dec}} & =
\sup\limits_{\substack{a_j\in M^\prime \\ \| a_j\|\leqslant 1}} \bigg\| \sum_j a_j x_j \bigg\| \leqslant \sup
\Bigg\{ \bigg\| \sum_j \Phi(U_j)x_j \bigg\| ; \ \Phi:C^*(F)\rightarrow M^\prime \ \mbox{\rm ucp} \Bigg\}
\\
& \leqslant \bigg\| \sum_j U_j \otimes x_j \bigg\|_{C^*(F) \underset{\operatorname{nor}}{\otimes} M} \leqslant 1 .
\end{split}
\end{equation*}

Let $\rho:M\rightarrow \BB(\HH)$ be a representation. By \eqref{eq3} and Proposition 1.3, \cite{Haa} we have
\begin{equation*}
\sum\limits_{\substack{b_j\in \rho(M)^\prime \\ \| b_j\| \leqslant 1}} \bigg\| \sum_j b_j \rho(x_j)\bigg\| =
\big\| \rho T_{\{ x_j\}_j} \big\|_{\operatorname{dec}} \leqslant \big\| T_{\{ x_j\}_j} \big\|_{\operatorname{dec}} \leqslant 1,
\end{equation*}
and therefore $\big\| \sum\limits_j U_j \otimes x_j \big\|_{C^*(F) \underset{\operatorname{max}}{\otimes} M} \leqslant 1$. The arguments from
\cite{Pis} show that this implies \eqref{eq2}.

If $M$ is not $\sigma$-finite, one proves
\begin{equation*}
\bigg\| \sum_j U_j \otimes x_j \bigg\|_{C^*(F) \underset{\operatorname{max}}{\otimes} M} =
\bigg\| \sum_j U_j \otimes x_j \bigg\|_{C^*(F) \underset{\operatorname{nor}}{\otimes} M}
\end{equation*}
using a net $(p_\iota)_{\iota}$ of $\sigma$-finite projections such that $p_\iota \nearrow 1$ strongly.

Finally, let $A=C^*(F)/J$ be a $C^*$-algebra such that $\operatorname{id}_A :A\rightarrow B/J$
is liftable. The equality $A\underset{\operatorname{nor}}{\otimes} M =A\underset{\operatorname{max}}{\otimes} M$ follows as in (i)
using \eqref{eq2} and part (ii) in Lemma 3.
\end{proof}

\begin{remark}
{\rm (i) A different proof of \eqref{eq2} was obtained by Kirchberg and Wassermann (\cite{KW}).

(ii) By the previous proof we may assume only that $\operatorname{id}_A :A \rightarrow C^*(F)/J$ is
locally liftable (respectively liftable), where $F$ is a free group such that $A=C^*(F)/J$.}
\end{remark}

The next proof of ${\rm (iii)} \Rightarrow {\rm (i)}$ in Theorem B follows the ideas of
${\rm (vi)} \Rightarrow {\rm (i)}$ in Proposition 2.2, \cite{Kir1}. Nevertheless, it uses only the fact that
$C^*(F) \underset{\operatorname{min}}{\otimes} \BB(\HH) =C^*(F)\underset{\operatorname{max}}{\otimes} \BB(\HH)$ and does not appeal to the (local)
liftability of $C^*(F)$ (\cite{Kir2}, Lemma 3.3).

\begin{proposition}
Let $A$ be a $C^*$-algebra. The following are equivalent:

{\rm (i)} $\  A$ has {\rm LLP}.

{\rm (ii)} There exist a free group $F$ and a morphism $\pi$ from $B=C^*(F)$ onto $A$ such that
$\operatorname{id}_A:A \rightarrow B/\operatorname{Ker}\pi$ is locally liftable.

{\rm (iii)} $A\underset{\operatorname{min}}{\otimes} \BB(\HH) =A\underset{\operatorname{max}}{\otimes} \BB(\HH)$.
\end{proposition}

\begin{proof}
${\rm (i)} \Rightarrow {\rm (ii)}$ Obvious.

${\rm (ii)} \Rightarrow {\rm (iii)}$ Follows from Corollary 4 and its subsequent remark.

${\rm (iii)} \Rightarrow {\rm (i)}$ Let $\Phi:A\rightarrow B/J$ be a ucp map, $E\subset A$ be a finite-dimensional
operator system and $\pi:B\rightarrow B/J$ be the quotient morphism.

We can always assume $B=C^*(F)$. For, let $F$ be a free group such that there exists a morphism
$\pi_0$ from $C^*(F)$ onto $B$. If there exists a ucp map $\Psi:E\rightarrow C^*(F)$ such that
$\pi \pi_0 \Psi =\Phi\vert_E$, then $\widetilde{\Psi}=\pi_0 \Psi :E\rightarrow B$ is a ucp map and
$\pi \widetilde{\Psi} =\Phi \vert_E$.

In the remainder we take $B=C^*(F)$. The idea of proof is from Theorem 3.2, \cite{EH} (see also
Proposition 6.8, \cite{Was}). Fix a central approximate unit $\{ e_\iota\}_\iota$ for $B$ in $J$
(\cite{Ped}) and a linear self-adjoint map $\Psi:E\rightarrow B$ with $\pi\Psi =\Phi\vert_E$.
The map $\Psi_\iota (x)=(1-e_\iota)^{1/2} \Psi(x) (1-e_{\iota})^{1/2}$, $x\in E$, fulfills
$\pi\Psi_\iota=\Phi\vert_E$ and $\| \Psi_\iota \|_{\operatorname{cb}} \geqslant 1$ for all $\iota$.

The liftability of $\Phi\vert_E$ follows as at the end of the proof of Proposition 6.8 in \cite{Was},
once we prove that $\lim\limits_\iota \| \Psi_\iota \|_{\operatorname{cb}} =1$. Assume that
$\limsup\limits_\iota \| \Psi_\iota\|_{\operatorname{cb}} >1$. Passing eventually to a subnet, we get $\varepsilon >0$ such that
$\| \Psi_\iota \|_{\operatorname{cb}} \geqslant 1+\varepsilon$ for all $\iota$. Choose $n_\iota \in \N^*$ and
$x_\iota \in E\underset{\operatorname{min}}{\otimes} M_{n_\iota} (\C)$, $\| x_\iota \| \leqslant 1$ such that
\begin{equation*}
\big\| \big( \Psi_\iota \otimes \operatorname{id} \big) (x_\iota) \big\| \geqslant 1+\frac{\varepsilon}{2} .
\end{equation*}

Let $C=\underset{\iota}{\oplus} M_{n_\iota} (\C)$. Set $x=(x_\iota)_\iota\in \underset{\iota}{\oplus}
\big( E \underset{\operatorname{min}}{\otimes} M_{n_\iota} (\C)\big) =E\underset{\operatorname{min}}{\otimes} C$. Then
\begin{equation}\label{eq4}
\big\| \big(\Psi_\iota \otimes \operatorname{id}_C\big) (x) \big\|_{B\underset{\operatorname{min}}{\otimes} C} \geqslant
\big\| \big(  \Psi_\iota \otimes \operatorname{id}_{n_\iota} \big) (x_\iota) \big\| \geqslant 1+\frac{\varepsilon}{2} \quad \mbox{\rm for all $\iota$}
\end{equation}
and there exists $x\in E\odot C$, $\| x\|_{A\underset{\operatorname{min}}{\otimes} C} \leqslant 1$ such that \eqref{eq4} is fulfilled.
Using the basic properties of central approximate units we get as in Proposition 6.8, \cite{Was}:
\begin{equation}\label{eq5}
\lim_\iota \big\| \big( \Psi_\iota \otimes \operatorname{id}_C \big) (x) \big\|_{B\underset{\operatorname{min}}{\otimes} C} =
\underset{a\in J\underset{\operatorname{\operatorname{min}}}{\otimes} C}{\inf} \big\| \big( \Psi \otimes \operatorname{id}_C \big) (x) +a\big\|_{B\underset{\operatorname{min}}{\otimes} C} .
\end{equation}

The existence of a conditional expectation from $\BB(\HH)$ onto $C$ and the equalities
$A\underset{\operatorname{max}}{\otimes} \BB(\HH)= A\underset{\operatorname{min}}{\otimes} \BB(\HH)$ and
$B\underset{\operatorname{max}}{\otimes} \BB(\HH) =B\underset{\operatorname{min}}{\otimes} \BB(\HH)$ imply
$A\underset{\operatorname{max}}{\otimes} C=A\underset{\operatorname{min}}{\otimes} C$ and respectively
$B\underset{\operatorname{max}}{\otimes} C=B\underset{\operatorname{min}}{\otimes} C$. By Lemma 1 we also have
$J\underset{\operatorname{max}}{\otimes} C=J\underset{\operatorname{min}}{\otimes} C$. These equalities together with \eqref{eq5} yield
\begin{equation*}
\lim_\iota \big\| \big( \Psi_\iota \otimes \operatorname{id}_C \big) (x) \big\|_{B\underset{\operatorname{min}}{\otimes} C} =
\underset{a\in J\underset{\operatorname{max}}{\otimes} C}{\inf} \big\| \big( \Psi\otimes \operatorname{id}_C \big) (x)+a\big\|_{B\underset{\operatorname{max}}{\otimes} C} .
\end{equation*}

The last equality and the exactness of
\begin{equation*}
0\rightarrow J\underset{\operatorname{max}}{\otimes} C \rightarrow B\underset{\operatorname{max}}{\otimes} C \rightarrow A\underset{\operatorname{max}}{\otimes} C \rightarrow 0
\end{equation*}
finally imply
\begin{equation*}
\begin{split}
\lim_\iota \big\| \big( \Psi_\iota \otimes \operatorname{id}_C \big) (x) \big\|_{B\underset{\operatorname{min}}{\otimes} C} &
= \big\| (\pi\otimes \operatorname{id}_C )(\Psi \otimes \operatorname{id}_C)(x)\big\|_{B/J \underset{\operatorname{min}}{\otimes} C} \\
& =\big\| (\Phi \otimes \operatorname{id}_C \big) (x) \big\|_{B/J \underset{\operatorname{max}}{\otimes} C} \\
& \leqslant \| x\|_{A\underset{\operatorname{max}}{\otimes} C} =\| x\|_{A\underset{\operatorname{min}}{\otimes} C} \leqslant 1,
\end{split}
\end{equation*}
which contradicts \eqref{eq4}.
\end{proof}

In the second part we study the behaviour of these properties with respect to the full free $C^*$-product,
using the result from \cite{Bo1} (see also \cite{Bo2} and \cite{Boz}).

\begin{proposition}
Let $(A_i)_{i\in I}$ be a family of unital $C^*$-algebras such that $\operatorname{id}_{A_i}$ is liftable for any $i\in I$.
Then the identity map on $A=\underset{i\in I}{\ast} A_i$, the full free $C^*$-algebra with amalgamation over $\C 1$, is liftable.
In particular $A\underset{\operatorname{nor}}{\otimes} M =A\underset{\operatorname{max}}{\otimes} M$ for any von Neumann algebra $M$.
\end{proposition}

\begin{proof}
Let $B$ be a $C^*$-algebra and $\pi$ a morphism from $B$ onto $A$. Since $\operatorname{id}_{A_i}$ is liftable for all $i$,
there exist ucp maps $\Phi_i :A_i \rightarrow B_i =\pi^{-1}(A_i)\subset B$ with $\pi\Phi_i =\operatorname{id}_{A_i}$.

Let $\varphi_i$ be a state on $A_i$ and $A_i^0$ be the kernel of $\varphi_i$. The algebraic free product with amalgamation over $\C 1$,
$\AA=\underset{i\in I}{\circledast} A_i$, is isomorphic with $\C 1 \oplus \underset{i_1\neq \cdots \neq i_n}{\bigoplus} A^0_{i_1}
\otimes \cdots \otimes A_{i_n}^0$ (as vector spaces). By \cite{Bo1}, the unital linear map $\Phi:\AA \rightarrow B$ defined by
$\Phi(a_1 \cdots a_n)=\Phi_{i_1}(a_1) \cdots \Phi_{i_n}(a_n)$, $a_j \in A_{i_j}^0$, $i_1\neq \cdots \neq i_n$,
extends to a ucp map $\Phi:A\rightarrow B$ such that $\pi\Phi=\operatorname{id}_A$.
\end{proof}

\begin{cor}
Let $(A_i)_{i\in I}$ be a family of unital $C^*$-algebras such that $\operatorname{id}_{A_i}$ is liftable for any $i\in I$.
Then $\Big( \underset{i\in I}{\ast} A_i \Big) \underset{\operatorname{nor}}{\otimes} M=\Big( \underset{i\in I}{\ast} A_i \Big) \underset{\operatorname{max}}{\otimes} M$
for any von Neumann algebra $M$.
\end{cor}

In order to prove the ``local" version of Proposition 7 we have to proceed more carefully.
The following lemma and its corollary are standard (\cite{EH}, see also \cite{Was}).

\begin{lemma}
Let $A$ and $B$ be $C^*$-algebras, $X_1 \subset X_2$ finite dimensional operator systems in $A$ and $\pi$
a morphism from $B$ onto $A$. If there exist ucp liftings
$\Psi_j :X_j \rightarrow B$ of $\operatorname{id}_{X_j}$, $j=1,2$, and $\Psi_1$ extends to a ucp map
$\widetilde{\Psi}_1 :X_2 \rightarrow B$, then for any $\varepsilon >0$ there exists a ucp map
$\Psi:X_2 \rightarrow B$ such that $\pi\Psi =\operatorname{id}_{X_2}$ and
$\big\| \Psi_1 -\Psi\vert_{X_1} \big\| < \varepsilon$.
\end{lemma}

\begin{proof}
Let $J$ be the kernel of the quotient morphism $\pi:B \rightarrow A=B/J$ and $\{ e_\iota\}_\iota$ be a
central approximate unit for $A$ in $J$. Since $\Psi_1 (x)-\Psi_2 (x) \in J$, $x\in X_1$, we obtain:
\begin{equation*}
\lim_\iota \big\| \Psi_1 (x)-(1-e_\iota)^{\frac{1}{2}} \Psi_2(x) (1-e_\iota)^{\frac{1}{2}} -e_\iota^{\frac{1}{2}}
\Psi_1(x) e_\iota^{\frac{1}{2}} \big\| =0 \quad \mbox{\rm for all $x\in X_1$.}
\end{equation*}

Since $X_1$ is finite dimensional, a standard compactness argument yields an index $\iota_0$
such that $e=e_{\iota_0}$ satisfies:
\begin{equation}\label{eq6}
\big\| \Psi_1 (x) -(1-e)^{\frac{1}{2}} \Psi_2(x) (1-e)^{\frac{1}{2}} -e^{\frac{1}{2}} \Psi_1(x) e^{\frac{1}{2}} \big\| <
\varepsilon \| x\| \quad \mbox{\rm for all $x\in X_1$.}
\end{equation}

Setting $\Psi(x)=(1-e)^{1/2} \Psi_2(x)(1-e)^{1/2}+e^{1/2} \widetilde{\Psi}_1 (x) e^{1/2}$, we obtain a ucp map
$\Psi:X_2 \rightarrow B$ such that $\pi\Psi=\pi\Psi_2 =\operatorname{id}_{X_2}$. The inequality
$\big\| \Psi\vert_{X_1}-\Psi_1 \big\| <\varepsilon$ follows from \eqref{eq6}.
\end{proof}

\begin{cor}
Let $A$ and $B$ be $C^*$-algebras, $\pi$ a morphism from $B$ onto $A$ and $\{ X_n\}_{n\geqslant 1}$ an increasing sequence of finite dimensional
operator systems with the property that for all $n\geqslant 1$ there exist ucp maps $\Phi_n :X_n \rightarrow B$ and
$\widetilde{\Phi}_n :X_{n+1} \rightarrow B$ such that $\pi\Phi_n =\operatorname{id}_{X_n}$ and
$\widetilde{\Phi}_n \vert_{X_n}=\Phi_n$. Let $X$ be the norm closure of $\underset{n\geqslant 1}{\cup} X_n$ in $A$.
Then, there exists a ucp map $\Phi:X\rightarrow B$ such that $\pi\Phi=\operatorname{id}_X$.
\end{cor}

\begin{proof}
By the previous lemma there exist ucp maps $\Psi_n :X_n \rightarrow B$ such that $\pi\Psi_n =\operatorname{id}_{X_n}$ and
$\big\| \Psi_{n+1}\vert_{X_n} -\Psi_n \big\| \leqslant 2^{-n}$. For any $x\in \underset{n\geqslant 1}{\cup} X_n$, the
sequence $\{ \Psi_n(x)\}_{n\geqslant 1}$ is Cauchy and defines a ucp lifting $\Phi(x)=\lim\limits_n \Psi_n(x)$ which
extends to a ucp map on $X$.
\end{proof}

\begin{prop11}
Let $\{ A_i\}_{i\in I}$ be a family of $C^*$-algebras such that $\operatorname{id}_{A_i}$ is locally
liftable for all $i\in I$. Then $\operatorname{id}_{\underset{i\in I}{\ast} A_i}$ is locally liftable.
\end{prop11}

\begin{proof}
For each $i\in I$, fix a state $\varphi_i$ on $A_i$ and denote by $A_i^0$ the kernel of $\varphi_i$.
Let $A=\underset{i\in I}{\ast} A_i$ be the free $C^*$-product with amalgamation over $\C 1$, $\pi$ be a morphism from $B$ onto $A$, where $B$ is a unital
$C^*$-algebra acting faithfully on a Hilbert space $\HH$. Let $X\subset A$ be a finite dimensional operator system.
It is plain to check that there exists a countable subset $I_0 \subset I$ such that
$X\subset \underset{i\in I_0}{\ast} A_i$. Moreover, there exists
an increasing sequence of finite dimensional operator systems
$\{ X_n\}_{n\geqslant 1}$ with $X_n \subset \C 1 \oplus \underset{\substack{i_1\neq\cdots \neq i_m \\ m\leqslant n}}{\bigoplus} A_{i_1}^0 \otimes \cdots
\otimes A_{i_m}^0$ and $X \subset \overline{\underset{n\geqslant 1}{\cup} X_n}$.

According to the previous corollary, it suffices to prove that for any $n\geqslant 1$, there exist ucp maps
$\Phi_n :X_n \rightarrow B$ and $\widetilde{\Phi}_n :X_{n+1} \rightarrow B$ such that $\pi\Phi_n =\operatorname{id}_{X_n}$ and
$\widetilde{\Phi}_n \vert_{X_n}=\Phi_n$.

Fix $n\geqslant 1$. Then, there exist a finite set $F\subset I_0$ and finite dimensional operator systems
$Y_i=\C 1+Y_i^0 \subset A_i$, $i\in I$, such that for all $k\leqslant n+1$:
\begin{equation*}
X_k \subset Y_{n+1} =\C 1 +\sum\limits_{\substack{i_1\neq\cdots\neq i_m \\ i_j \in F, m\leqslant n+1}} X_{i_1}^0 \cdots X_{i_m}^0 =
\C 1 \oplus \bigoplus\limits_{\substack{i_1\neq\cdots \neq i_m \\ i_j \in F, m\leqslant n+1}}
X_{i_1}^0 \otimes \cdots \otimes X_{i_m}^0 .
\end{equation*}

Each $\operatorname{id}_{X_i}$ is liftable, hence there exist ucp maps $\Psi_i:X_i \rightarrow B$ such that
$\pi\Psi_i=\operatorname{id}_{X_i}$, $i\in I$. They extend by Arveson's Theorem (\cite{Arv}) to ucp maps
$\widetilde{\Psi}_i :A_i \rightarrow \BB(\HH)$. The unital map defined on the algebraic free product
$\underset{i\in I}{\circledast} A_i$ by $\Psi(a_1\cdots a_m)=\widetilde{\Psi}_{i_1}(a_1)\cdots \widetilde{\Psi}_{i_m}(a_m)$,
$a_j \in A_{i_j}^0$, $i_1\neq \cdots \neq i_m$, extends (\cite{Bo1},\cite{Bo2}) to a ucp map
$\Psi: \underset{i\in I}{\ast} A_i \rightarrow \BB(\HH)$. By the very definition of $\Psi$ we have
$\Psi(Y_{n+1}) \subset B$ and $\pi\Psi \vert_{Y_{n+1}} =\operatorname{id}_{Y_{n+1}}$. Since $Y_{n+1}$ contains both $X_n$ and $X_{n+1}$, we may take
$\Phi_n =\Psi\vert_{X_n}$, $\widetilde{\Phi}_n =\Psi\vert_{X_{n+1}}$, and apply Corollary 10.
\end{proof}

\begin{cor12}
Let $(A_i)_{i\in I}$ be a family of unital $C^*$-algebras such that $\operatorname{id}_{A_i}$ is locally liftable for all $i\in I$.
Then the $C^*$-free product $\underset{i\in I}{\ast} A_i$ with amalgamation over $\C 1$ has {\rm LLP}. In particular
$\Big( \underset{i\in I}{\ast} A_i \Big) \underset{\operatorname{min}}{\otimes} \BB(\HH) =
\Big( \underset{i\in I}{\ast} A_i \Big) \underset{\operatorname{max}}{\otimes} \BB(\HH)$.
\end{cor12}

In the last part of this note we prove that quasidiagonality is preserved under full free $C^*$-products with amalgamation over $\C 1$.

\begin{prop13}
If $A_1$ and $A_2$ are quasidiagonal $C^*$-algebras, then $A=A_1 \underset{\C 1}{\ast} A_2$ is quasidiagonal.
\end{prop13}

\begin{proof}
By Theorem 1, \cite{Voi}, $A$ is quasidiagonal if and only if for any $\varepsilon >0$ and any finite subset $F\subset A$,
there exist a finite dimensional $C^*$-algebra $B$ and a ucp map $\Phi:A\rightarrow B$ such that
\begin{equation}\label{eq7}
\| \Phi (a)\| > \| a\| -\varepsilon \quad \mbox{\rm for $a\in F$}
\end{equation}
\begin{equation}\label{eq8}
\| \Phi (b)\Phi(c) -\Phi (bc) \| <\varepsilon \quad \mbox{\rm for $b,c\in F$.}
\end{equation}

We can make the following two straightforward reductions:

(I) assume that $F\subset \AA=A_1 \underset{\C 1}{\circledast} A_2 =\C 1 \oplus \underset{i_1 \neq\cdots\neq i_n}{\bigoplus}
A_{i_1}^0 \otimes \cdots \otimes A_{i_n}^0$ as vector spaces, where $A_i^0=\operatorname{Ker} \varphi_i$ and $\varphi_i$ denotes a (fixed) state
on $A_i$, $i=1,2$.

(II) prove for any $\varepsilon >0$, $F\subset \AA$ finite subset and $a\in F$ the existence of a
finite dimensional $C^*$-algebra $B_a$ and of a ucp map $\Phi_a:A\rightarrow B_a$ such that
\begin{equation*}
\begin{split}
& \| \Phi_a (a)\| > \| a\|-\varepsilon \\
& \| \Phi_a(b)\Phi_b (c)-\Phi_a(bc)\| <\varepsilon \quad \mbox{\rm for $b,c\in F$.}
\end{split}
\end{equation*}
For, note that \eqref{eq7} and \eqref{eq8} are fulfilled by $\Phi:A\rightarrow B =\underset{a\in F}{\oplus} B_a$,
$\Phi(x)=\underset{a\in F}{\oplus} \Phi_a(x)$.

To prove (II), fix a faithful representation $\pi:A\rightarrow \BB (\HH)$, a finite set $F\subset \AA$
and $\delta >0$. Let $\xi_0$ be a unit vector in $\HH$ such that $\| \pi(a)\xi_0\| > \| a\| -\delta$.
Each $b\in F$ decomposes into a finite sum
\begin{equation}\label{eq9}
b=\alpha_0 1 +\sum_{i_1\neq\cdots\neq i_n} b_{i_1} \cdots b_{i_n},\quad
\alpha_0 \in \C, \ b_j\in A_{i_j}^0,\  i_1 \neq \cdots \neq i_n .
\end{equation}

Denote by $F_j$, $j=1,2$, the set of such elements of $A_j^0$ which appear in the decomposition of
elements from $F$. The sets $F_j$ are finite, for we convene to choose only one such decomposition for each $b\in F$.
Consider also the finite dimensional subspace $\HH_0$ of $\HH$ spanned by $\xi_0$ and vectors
$\pi(b_{i_1}\cdots b_{i_n})\xi_0$, with $b\in F$ and $b_{i_1}\cdots b_{i_n}$ as in \eqref{eq9}.

The representations $\pi_j=\pi\vert_{A_j}$, $j=1,2$ are faithful, hence $\pi_j(A_j)$ are quasidiagonal and
there exist $\HH_1 \subset \HH_2$ subspaces of $\HH$ such that $d_j=\operatorname{dim} \HH_j <\infty$,
$\HH_0 \subset \HH_1$ and $\big\| \big[ P_{\HH_j} ,\pi_j (x)\big] \big\| <\delta$ for all $x\in F_j$.
Since $\HH_1 \subset \HH_2$, there exist mutually orthogonal projections
$e_1=P_{\HH_1},e_2,\ldots,e_{d_2}$ and mutually orthogonal projections $f_1=P_{\HH_2},f_2,\ldots,f_{d_1}$
such that $e_i f_j=f_j e_i$ and $\sum\limits_{i=1}^{d_2} e_i =\sum\limits_{j=1}^{d_1} f_j =e$. Let
$v_1,\ldots,v_{d_2}$ and $w_1,\ldots,w_{d_1}$ be partial isometries such that $v_i^* v_i=e_1$,
$v_iv_i^*=e_i$, $w_j^* w_j=f_1$, $w_j w_j^*=f_j$. Set $\KK=e\HH$ and consider the ucp maps
$\Phi_j :A_j \rightarrow \BB(\KK)$ defined by
\begin{equation*}
\begin{split}
& \Phi_1 (x)=\sum_{i=1}^{d_2} v_i \pi_1(x)v_i^*,\quad x\in A_1 \\
& \Phi_2(y) =\sum_{j=1}^{d_1} w_j \pi_2 (y) w_j^* ,\quad y\in A_2 .
\end{split}
\end{equation*}

We have
\begin{equation}\label{eq10}
\| \Phi_j (x)\Phi_j (y)-\Phi_j (xy)\| <\delta\quad \mbox{\rm for all $x,y\in F_j$.}
\end{equation}

By \cite{Bo1}, the unital linear map $\Phi:\AA =\C 1 \oplus \underset{i_1\neq\cdots\neq i_n}{\bigoplus}
A_{i_1}^0 \otimes \cdots \otimes A_{i_n}^0 \rightarrow B_a =\BB(\KK)$, $\Phi(x_1\cdots x_m)=\Phi_{i_1}(x_1) \cdots \Phi_{i_m} (x_m)$,
$x_j \in A_{i_j}^0$, $i_1\neq\cdots\neq i_m$, extends to a ucp map $\Phi:A\rightarrow \BB(\KK)$. We have
\begin{equation*}
\big\| \Phi(a) (\xi_0 \oplus \underbrace{0_{\HH_1} \oplus \cdots \oplus 0_{\HH_1}}_{d_2 -1} )\big\| =\| \pi (a)\xi_0 \| > \| a\| -\delta .
\end{equation*}

Moreover, by \eqref{eq10} and the definition of $\Phi$ we obtain
\begin{equation*}
\| \Phi(b)\Phi(c) -\Phi(bc)\| < M\delta ,
\end{equation*}
where $M$ is a positive integer which depends only on $F$, $F_1$ and $F_2$. Assuming $\delta <\frac{\varepsilon}{M}$ from the beginning,
the proof is complete.
\end{proof}

The following is a plain consequence of the proof of Proposition 13.

\begin{cor14}
If $(A_i)_{i\in I}$ is a family of quasidiagonal $C^*$-algebras, then $\underset{i\in I}{\ast} A_i$ is quasidiagonal.
\end{cor14}

\bigskip

\textbf{\small Acknowledgments}
I am grateful to G. Pisier for sending me an early draft of the paper (\cite{Pis}) and to G. Pisier and S. Wassermann for useful conversations.

\medskip

\begin{center}
\emph{Research supported by an EPSRC Advanced Fellowship}
\end{center}

\end{document}